\documentclass[11pt]{amsart}
\usepackage{mathrsfs,latexsym,amsfonts,amssymb}
\usepackage{hyperref}
\newtheorem{theorem}{Theorem}[section]
\newtheorem{lemma}[theorem]{Lemma}
\newtheorem{corollary}[theorem]{Corollary}

\theoremstyle{definition}
\newtheorem{definition}[theorem]{Definition}
\newtheorem{proposition}[theorem]{Proposition}
\theoremstyle{remark}

\begin{document}

\title[Embedding into connected, locally connected topological gyrogroups]
{Embedding into connected, locally connected topological gyrogroups}

\author{Jingling Lin}
\address{(Jingling Lin): School of mathematics and statistics,
Minnan Normal University, Zhangzhou 363000, P. R. China}
\email{jinglinglin1995@163.com}

\author{Meng Bao}
\address{(Meng Bao): College of Mathematics, Sichuan University, Chengdu 610064, P. R. China}
\email{mengbao95213@163.com}

\author{Fucai Lin*}
\address{(Fucai Lin): 1. School of mathematics and statistics,
Minnan Normal University, Zhangzhou 363000, P. R. China; 2. Fujian Key Laboratory of Granular Computing and Applications,
Minnan Normal University, Zhangzhou 363000, P. R. China}
\email{linfucai@mnnu.edu.cn; linfucai2008@aliyun.com}

\thanks{This work is supported by the Key Program of the Natural Science Foundation of Fujian Province (No: 2020J02043), the National
Natural Science Foundation of China (Grant Nos. 11571158, 12071199), the Institute of Meteorological Big Data-Digital Fujian and Fujian Key Laboratory of Data Science and Statistics.\\
*corresponding author}

%insert keywords
%insert subject class

%\date{\today}
\begin{abstract}
In this paper, it is proved that every topological gyrogroup $G$ is topologically groupoid isomorphic to a closed subgyrogroup of a connected, locally connected topological gyrogroup $G^{\bullet}$.
\end{abstract}

\maketitle
\section{Introduction}
In the study of the $c$-ball of relativistically admissible velocities with Einstein velocity addition, Ungar discovered that the seemingly structureless Einstein addition of relativistically admissible velocities possesses a rich grouplike structure and he posed the concept of gyrogroups. In 2017, Atiponrat \cite{AW} gave the concept of topological gyrogroups and investigated some basic properties of topological gyrogroups. A topological gyrogroup is a gyrogroup $G$ endowed with a topology such that the binary operation $\oplus: G\times G\rightarrow G$ is jointly continuous and the inverse mapping $\ominus (\cdot): G\rightarrow G$, i.e. $x\rightarrow \ominus x$, is also continuous. Then Cai, Lin and He in \cite{CZ} proved that every topological gyrogroup is a rectifiable space, which deduced that every first-countable topological gyrogroup is metrizable. Moreover, Wattanapan, Atiponrat and Suksumran \cite{WAS2020} proved that every locally compact Hausdorff topological gyrogroup can be embedded in a completely regular topological group as a twisted subset. Since it is obvious that each topological group is a topological gyrogroup, it is natural to consider whether some well-known results of topological groups can be extended to topological gyrogroups. In this paper, we show that every topological gyrogroup $G$ is topologically groupoid isomorphic to a closed subgyrogroup of a connected, locally connected topological gyrogroup $G^{\bullet}$, and we discuss some properties of topological gyrogroups $G$ and $G^{\bullet}$.

\smallskip
\section{Preliminaries}

Throughout this paper, all topological spaces are assumed to be $T_{1}$, unless otherwise is explicitly stated. Let $\mathbb{N}$ be the set of all positive integers and $\omega$ the first infinite ordinal. Next we recall some definitions and facts.

Let $G$ be a nonempty set, and let $\oplus:G\times G\rightarrow G$ be a binary operation on $G$. Then $(G, \oplus)$ is called a groupoid. A function $f$ from a groupoid $(G_{1}, \oplus_{1})$ to a groupoid $(G_{2}, \oplus_{2})$ is said to be a groupoid homomorphism if $f(x_{1}\oplus_{1} x_{2})= f(x_{1})\oplus_{2} f(x_{2})$ for any $x_{1}, x_{2}\in G_{1}$. In addition, a bijective groupoid homomorphism from a groupoid $(G, \oplus)$ to itself will be called a groupoid automorphism. We will write $Aut(G, \oplus)$ for the set of all automorphisms of a groupoid $(G, \oplus)$.

\begin{definition}\cite{UA}
Let $(G, \oplus)$ be a groupoid. The system $(G,\oplus)$ is called a {\it gyrogroup}, if its binary operation satisfies the following conditions:

\smallskip
(G1) There exists a unique identity element $0\in G$ such that $0\oplus a=a=a\oplus0$ for all $a\in G$.

\smallskip
(G2) For each $x\in G$, there exists a unique inverse element $\ominus x\in G$ such that $\ominus x \oplus x=0=x\oplus (\ominus x)$.

\smallskip
(G3) For all $x, y\in G$, there exists $\mbox{gyr}[x, y]\in \mbox{Aut}(G, \oplus)$ with the property that $x\oplus (y\oplus z)=(x\oplus y)\oplus \mbox{gyr}[x, y](z)$ for all $z\in G$.

\smallskip
(G4) For any $x, y\in G$, $\mbox{gyr}[x\oplus y, y]=\mbox{gyr}[x, y]$.
\end{definition}

Notice that a group is a gyrogroup $(G,\oplus)$ such that $\mbox{gyr}[x,y]$ is the identity function for all $x, y\in G$. Then the {\it gyrogroup cooperation} $\boxplus$ is defined by the equation $$ (\divideontimes)a\boxplus b=a\oplus \mbox{gyr}[a,\ominus b](b) \mbox{ for each } a,b\in G.$$ Replacing $b$ by $\ominus b$ in $(\divideontimes)$, along with Identity $(\divideontimes)$ we have the identity $$a\boxminus b=a\ominus \mbox{gyr}[a,b](b) \mbox{ for all }a, b\in G,$$ where we use the obvious notation, $a \boxminus b = a \boxplus (\ominus b)$.

\begin{lemma}{\rm (\cite{UA})}\label{a}
Let $(G, \oplus)$ be a gyrogroup. Then for any $x, y, z\in G$, we obtain the following:

\begin{enumerate}
\smallskip
\item $(\ominus x)\oplus (x\oplus y)=y$. \ \ \ (left cancellation law)

\smallskip
\item $(x\oplus (\ominus y))\oplus gyr[x, \ominus y](y)=x$. \ \ \ (right cancellation law)

\smallskip
\item $(y\ominus x)\boxplus x=y$.

\smallskip
\item $(y\boxplus (\ominus x))\oplus x=y$.

\smallskip
\item $gyr[x, y](z)=\ominus (x\oplus y)\oplus (x\oplus (y\oplus z))$.
\end{enumerate}
\end{lemma}

About the gyrogroup cooperation, we have the following Proposition~\ref{mt2}.

\begin{proposition}\label{mt1}
Let $(G, \oplus)$ be a gyrogroup. Then for any $x,y,z\in G$, $y=z$ if and only if $y\oplus x=z\oplus x$.
\end{proposition}

\begin{proof}
The necessity is trivial. On the other hand, it follows from the right cancellation law that $$y=(y\oplus x)\oplus \mbox{gyr}[y,x](\ominus x)=(y\oplus x)\oplus \mbox{gyr}[y\oplus x,x](\ominus x)$$ and $$z=(z\oplus x)\oplus \mbox{gyr}[z,x](\ominus x)=(z\oplus x)\oplus \mbox{gyr}[z\oplus x,x](\ominus x).$$ By the hypothesis, $y\oplus x=z\oplus x$, so it is clear that $y=z$.
\end{proof}

\begin{proposition}\label{mt2}
Let $(G, \oplus)$ be a gyrogroup. Then for any $x,y,z\in G$, $y=z$ if and only if $x\boxplus y=x\boxplus z$.
\end{proposition}

\begin{proof}

On the one hand, if $y=z$, it is obvious that $x\boxplus y=x\boxplus z$.

On the other hand, if $x\boxplus y=x\boxplus z$, since $$x\boxplus y=x\oplus \mbox{gyr}[x,\ominus y](y)=x\oplus(\ominus(x\ominus y)\oplus x)$$ and $$x\boxplus z=x\oplus \mbox{gyr}[x,\ominus z](z)=x\oplus(\ominus(x\ominus z)\oplus x),$$ we obtain that $y=z$ by Lemma \ref{a} and Proposition \ref{mt1}.
\end{proof}

The definition of a subgyrogroup is given as follows.

\begin{definition}\cite{ST}
Let $(G,\oplus)$ be a gyrogroup. A nonempty subset $H$ of $G$ is called a {\it subgyrogroup}, denoted
by $H\leq G$, if $H$ forms a gyrogroup under the operation inherited from $G$ and the restriction of $\mbox{gyr}[a,b]$ to $H$ is an automorphism of $H$ for all $a,b\in H$.
\end{definition}

\begin{definition}\cite{AW}
A triple $(G, \tau, \oplus)$ is called a topological gyrogroup if
and only if

\smallskip
(1) $(G,\tau)$  is a topological space;

\smallskip
(2) $(G, \oplus)$ is a gyrogroup;

\smallskip
(3) The binary operation $\oplus:G\rightarrow G$ is continuous where $G\times G$ is endowed with the product topology and the operation of taking the inverse $\ominus(\cdot):G\rightarrow G,$ i.e. $x\rightarrow \ominus x$, is continuous.
\end{definition}

It was proved in \cite[Examples 2 and 3]{AW} that each of the M\"{o}bius gyrogroup or the Einstein gyrogroup equipped with standard topology is a topological gyrogroup which is not a topological group.

\begin{proposition}\cite{AW}
Let $(G, \tau , \oplus)$ be a topological gyrogroup, and let $U$ be a neighborhood of the identity element $\mathbf{0}$. Then the following are true:

\smallskip
(1) There is an open symmetric neighborhood $V$ of $\mathbf{0}$ such that $V \subseteq U$ and $V \oplus V \subseteq U$.

\smallskip
(2)  There is an open neighborhood $V$ of $\mathbf{0}$ such that $\ominus V \subseteq U$
\end{proposition}

 \smallskip
\section{embedded into connected, locally connected topological gyrogroups}
Let $(G,\tau ,\oplus)$ be a topological gyrogroup.  Consider the set $G^{\bullet}$ of all functions $f$ on $J = [0, 1)$ with values in $G$ such that, for some sequence $0 = a_{0} < a_{1} <\ldots< a_{n} = 1$, the function $f$ is constant on $[a_{k}, a_{k+1})$ for each $k = 0, \ldots, n-1$.  For each $f,g,h\in G^{\bullet}$ and $x\in J$, let us define a binary operation $\oplus^{\bullet}$ on $G^{\bullet}$ by $$(f\oplus^{\bullet}g)(x)=f(x)\oplus g(x),$$ and $$\mbox{gyr}[f,g](h)=\ominus^{\bullet}(f\oplus^{\bullet} g)\oplus^{\bullet} (f\oplus^{\bullet}(g\oplus^{\bullet}h)).$$ Then every element $f\in G^{\bullet}$ has a unique inverse $\ominus^{\bullet}f\in G^{\bullet}$ defined by $(\ominus^{\bullet}f)(r)=\ominus(f (r))$ for every $r\in J$. Let us show that $(G^{\bullet}, \oplus^{\bullet})$ is a gyrogroup with identity $\mathbf{0}^{\bullet}$, where $\mathbf{0}^{\bullet}(r)=0$ for each $r\in J$. Actually, for any $f,g,h\in G$, we claim that $$f\oplus^{\bullet}(g\oplus^{\bullet} h)=(f\oplus^{\bullet} g)\oplus^{\bullet}\mbox{gyr}[f,g](h) \mbox{ and }\mbox{gyr}[f\oplus^{\bullet}g,g]=\mbox{gyr}[f,g].$$
Indeed, for any $r\in J$,
\begin{eqnarray}
(f\oplus^{\bullet}(g\oplus^{\bullet}h))(r)&=&f(r)\oplus(g(r)\oplus h(r))\nonumber\\
&=&(f(r)\oplus g(r))\oplus \mbox{gyr}[f(r),g(r)](h(r))\nonumber\\
&=&((f\oplus^{\bullet} g)\oplus^{\bullet}\mbox{gyr}[f,g](h))(r),\nonumber
\end{eqnarray}

 \begin{eqnarray}
 \mbox{gyr}[f\oplus^{\bullet}g,g](h)(r)&=&\mbox{gyr}[f(r)\oplus g(r),g(r)](h(r))\nonumber\\
 &=&\mbox{gyr}[f(r),g(r)](h(r))\nonumber\\
 &=&\mbox{gyr}[f,g](h)(r).\nonumber
\end{eqnarray}
 Thus, $(G^{\bullet}, \oplus^{\bullet})$ is a gyrogroup. The elements of $G^{\bullet}$ are called step  functions.

In order to construct the topology on $G^{\bullet}$, we choose an open neighbourhood $V$ of $\mathbf{0}$ in $G$ and a real number $\varepsilon > 0$, and define a subset $O(V,\varepsilon)$ of $G^{\bullet}$ by $$O(V,\varepsilon)=\{f\in G^{\bullet}:\mu(\{r\in J:f(r)\not\in V\})<\varepsilon\},$$ where $\mu$ is the usual Lebesgue measure on $J$. Now we have the following important theorem.

\begin{theorem}\label{tttt}
The sets $O(V,\varepsilon)$ form a base of a Hausdorff topological gyrogroup topology at the identity of $G^{\bullet}$ such that $G^{\bullet}$ is a topological gyrogroup.
\end{theorem}

\begin{proof}
It suffices to prove that the family $$\mathcal{N}(\mathbf{0}^{\bullet})=\{O(V,\varepsilon):V\in\mathcal{N}(\mathbf{0}),\varepsilon >0\}$$  satisfies the nine conditions $(1)\sim(9)$ of Theorem 4.4 in \cite{JX}, where $\mathcal{N}(\mathbf{0})$ is a base for $G$ at $\mathbf{0}$.

\smallskip
(1) Take an arbitrary $U\in \mathcal{N}(\mathbf{0})$ and fix $\varepsilon >0$. Choose $V\in \mathcal{N}(\mathbf{0})$ with $V\oplus V\subseteq U$ and take $f,g\in O(V,\varepsilon/2)$. Then $\mu(\{r\in J: f(r)\oplus g(r)\not\in U\})<\varepsilon$,  this implies $O(V,\varepsilon/2)\oplus^{\bullet}O(V,\varepsilon/2)\subseteq O(U,\varepsilon)$.

\smallskip
(2) Let $O(U,\varepsilon)\in \mathcal{N}(\mathbf{0}^{\bullet})$ and that $f\in O(U,\varepsilon)$. Thus we can find real numbers $0=a_{0} < a_{1} < \ldots < a_{n} = 1 $ such that for any $k = 0, 1, \ldots , n-1$, $f$ is constant
on $J_{k}=[a_{k}, a_{k+1})$, and $f$ takes a value $x_{k}\in G$ on $J_{k}$. Since $f\in O(U,\varepsilon)$, it follows that the number
$\delta=\varepsilon -\mu(\{r\in J:f(r)\not\in U\})$ is positive. Take $V\in \mathcal{N}(\mathbf{0})$ such that if $0\leq k < n$ and
$x_{k} \in U$, then ${x_{k}}\oplus V\subseteq U$. We could show that $f\oplus O(V,\delta) \subseteq O(U,\varepsilon)$. Indeed, take an arbitrary $g\in O(V,\delta),$ then
\begin{eqnarray}
 \mu(\{r\in J:(f\oplus^{\bullet}g)(r)\not\in U\})&\leq& \mu(\{r\in J:x_{k}\oplus g(r)\not\in x_{k}\oplus V\})\nonumber\\
 &=& \mu(\{r\in J:g(r)\not\in V\})\nonumber\\
&<&\varepsilon,\nonumber
\end{eqnarray}
Thererore, $f\oplus^{\bullet}O(V,\delta)\subseteq O(U,\varepsilon).$

\smallskip
(3) Let $O(U,\varepsilon)\in \mathcal{N}(\mathbf{0}^{\bullet})$ and $f\in G^{\bullet}$ be arbitrary. Thus $f$, considered as a function from $J$ to $G$, choose only finitely many distinct values, say, $x_{1}, \cdots , x_{m}$. Takes  an element
$V\in \mathcal{N}(\mathbf{0})$ such that $(\ominus x_{i})\oplus (V\oplus x_{i})\subseteq U$, for each $i=1, \cdots ,m$. Let us show that $(\ominus^{\bullet} f)\oplus^{\bullet}(O(V, \varepsilon)\oplus^{\bullet}f)\subseteq O(U, \varepsilon).$ Indeed, take an arbitrary $g\in O(V,\varepsilon)$, we could show that $(\ominus^{\bullet} f)\oplus^{\bullet}(g\oplus^{\bullet}f)\in O(U, \varepsilon).$ Since
\begin{eqnarray}
 \mu(\{r\in J:((\ominus^{\bullet} f)\oplus^{\bullet}(g\oplus^{\bullet}f))(r)\not\in U\})&\leq& \mu(\{r\in J:(\ominus f(r))\oplus(g(r)\oplus f(r))\not\in U\})\nonumber\\
 &=& \mu(\{r\in J:(\ominus x_{i})\oplus(g(r)\oplus x_{i})\not\in U\})\nonumber\\
 &\leq& \mu(\{r\in J:(\ominus x_{i})\oplus(g(r)\oplus x_{i})\not\in (\ominus x_{i})\oplus (V\oplus x_{i})\})\nonumber\\
 &=&\mu(\{r\in J:g(r)\not\in V \})\nonumber\\
 &<&\varepsilon ,\nonumber
\end{eqnarray}
it follows that $(\ominus^{\bullet} f)\oplus^{\bullet}(O(V, \varepsilon)\oplus^{\bullet}f)\subseteq O(U, \varepsilon).$
\smallskip

(4) Take $O(U, \varepsilon_{1}),O(V, \varepsilon_{2})\in \mathcal{N}(\mathbf{0}^{\bullet})$, and put $W=U\cap V$ and
$\varepsilon= min\{\varepsilon_{1}, \varepsilon_{2}\}$. Hence it is obvious that $O(W, \varepsilon) \subseteq O(U, \varepsilon_{1}) \cap O(V, \varepsilon_{2})$.

\smallskip
(5) Let $O(U,\varepsilon)\in \mathcal{N}(\mathbf{0}^{\bullet})$ and $f,g\in G^{\bullet}$ be arbitrary. Then there exist real numbers $0=a_{0} < a_{1} < \cdots < a_{n} = 1 $ such that for each $k = 0, 1, \cdots , n-1$, $f,g$ is constant
on $J_{k}=[a_{k}, a_{k+1})$, and $f,g$ takes a value $x_{k}, y_{k}\in G$ on $J_{k}$, respectively. Choose $V\in \mathcal{N}(\mathbf{0})$ such that  $\mbox{gyr}[x_{k},y_{k}](V)\subseteq U$. We claim that $\mbox{gyr}[f,g](O(V,\varepsilon))\subseteq O(U,\varepsilon)$. Indeed, for any $h\in O(V,\varepsilon)$, we have
\begin{eqnarray}
 \mu(\{r\in J:\mbox{gyr}[f,g](h)(r)\not\in U\})&=& \mu(\{r\in J:\mbox{gyr}[f(r),g(r)](h(r))\not\in U\})\nonumber\\
 &=& \mu(\{r\in J:\mbox{gyr}[x_{k},y_{k}](h(r))\not\in U\})\nonumber\\
  &\leq& \mu(\{r\in J:\mbox{gyr}[x_{k},y_{k}](h(r))\not\in \mbox{gyr}[x_{k},y_{k}](V)\})\nonumber\\
 &=& \mu(\{r\in J:h(r)\not\in V\})\nonumber\\
 &<& \varepsilon.\nonumber
\end{eqnarray}
This proves $\mbox{gyr}[f,g](h)\subseteq O(U,\varepsilon)$, thus $\mbox{gyr}[f,g](O(V,\varepsilon))\subseteq O(U,\varepsilon).$

\smallskip
(6)  Let $O(U,\varepsilon)\in \mathcal{N}(\mathbf{0}^{\bullet})$ and $f\in G^{\bullet}$ be arbitrary. Then there exist real numbers $0=a_{0} < a_{1} < \cdots < a_{n} = 1 $ such that for each $k = 0, 1, \cdots , n-1$, $f$ is constant on $J_{k}=[a_{k}, a_{k+1})$, and $f$ takes a value $x_{k}\in G$ on $J_{k}$. Therefore, take $V\in \mathcal{N}(\mathbf{0})$ such that  $\bigcup_{v\in V}\mbox{gyr}[v,x_{k}](V)\subseteq U$. For every $g\in O(V,\varepsilon/2)$, let us show that $\mbox{gyr}[g,f](O(V,\varepsilon/2))\subseteq O(U,\varepsilon/2)$. Indeed, for every $h\in O(V,\varepsilon/2)$, there exist real numbers $0=b_{0} < b_{1} < \ldots < b_{m} = 1 $ such that $f,g,h$ are constant for any $k = 0, 1, \cdots , m-1, m\geq n$.

\begin{eqnarray}
 \mu(\{r\in J:\mbox{gyr}[g,f](h)(r)\not\in U\})&=& \mu(\{r\in J:\mbox{gyr}[g(r),f(r)]h(r)\not\in U\})\nonumber\\
 &=& \mu(\{r\in J: g(r)\in V , \mbox{gyr}[g(r),f(r)](h(r))\not\in U\})\nonumber\\
 &+& \mu(\{r\in J:g(r)\not \in V, \mbox{gyr}[g(r),f(r)](h(r))\not\in U\})\nonumber\\
  &\leq& \mu(\{r\in J: g(r)\in V , \mbox{gyr}[g(r),f(r)](h(r))\nonumber\\
  & & \ \ \ \ \ \ \ \ \ \ \ \ \ \ \ \ \ \ \ \ \ \ \not\in \mbox{gyr}[g(r),f(r)](V)\})+ \varepsilon/2\nonumber\\
 &=& \mu(\{r\in J:h(r)\not\in V\})+ \varepsilon/2\nonumber\\
 &<& \varepsilon/2+\varepsilon/2\nonumber\\
 &=&\varepsilon.\nonumber
\end{eqnarray}
Therefore, $\bigcup _{g\in O(V,\varepsilon/2)}\mbox{gyr}[g,f](O(V,\varepsilon/2))\subseteq O(U,\varepsilon)$.

\smallskip
(7) If $f\in G^{\bullet}$ and $f\not=\mathbf{0}^{\bullet}$, then the number $\varepsilon=\mu({r\in J:f(r)\not=\mathbf{0}})$ is positive. Take an element $U\in \mathcal{N}(\mathbf{0})$ such that $U\boxminus U$ does not contain any value of $f$ distinct from $\mathbf{0}$. We have to show that $f\not\in O(U, \varepsilon)\boxminus O(U, \varepsilon)$. Assume the contrary, there exist $g, h\in O(U, \varepsilon)$, real numbers $0=a_{0} < a_{1} < \ldots < a_{n} = 1 $, such that $f=g\boxminus h$, and for each $k = 0, 1, \ldots , n-1$, $f,g,h$ is constant
on $J_{k}=[a_{k}, a_{k+1})$. Then $f(r)=g(r)\boxminus h(r),r\in J$. since $g,h\in O(U, \varepsilon)$, $g(r')\boxminus h(r')\in U\boxminus U$ for some $r'\in J$, a contradiction. Thus $$\{\mathbf{0}^{\bullet}\}=\bigcap_{U\in\mathcal{N}(\mathbf{0})}(O(U, \varepsilon)\boxminus O(U, \varepsilon)).$$

\smallskip
(8) Let $O(U,\varepsilon)\in \mathcal{N}(\mathbf{0}^{\bullet})$ and $f\in G^{\bullet}$ be arbitrary. Choose $V\in \mathcal{N}(\mathbf{0})$ such that  $V\boxplus x\subseteq x\boxplus U$ for any $x\in G$, then $V\subseteq (x\boxplus U)\ominus x$ . We claim that $O(V,\varepsilon)\boxplus^{\bullet} f\subseteq f\boxplus^{\bullet}O(U,\varepsilon)$, it suffices to prove $O(V,\varepsilon)\subseteq (f\boxplus^{\bullet}O(U,\varepsilon))\ominus f$. For every $g\in O(V,\varepsilon)$, there exist real numbers $0=a_{0} < a_{1} < \ldots < a_{n} = 1 $ such that for each $k = 0, 1, \ldots , n-1$, $f,g$ are constant
on $J_{k}=[a_{k}, a_{k+1})$, and $f,g$ take values $x_{k}, y_{k}\in G$ on $J_{k}$, respectively. For the above division, we define the function $h:J\rightarrow G$ such that $$h(r)=\ominus\circ L_{\ominus f(r)}\circ \ominus \circ R_{\ominus f(r)}\circ L_{\ominus f(r)}\circ R_{\oplus f(r)}(g(r)),$$ this implies that $g(r)=((f\boxplus^{\bullet} h)\ominus^{\bullet} f)(r)$ for every $r\in J$. It remains to prove $h\in O(U,\varepsilon)$.
\begin{eqnarray}
 \mu(\{r\in J:h(r)\not\in U\})\nonumber&=& \mu(\{r\in J: (f(r)\boxplus h(r))\ominus f(r)\not\in (f(r)\boxplus U)\ominus f(r)\})\nonumber\\
 &=& \mu(\{r\in J: g(r)\not\in (f(r)\boxplus U)\ominus f(r)\})\nonumber\\
 &\leq& \mu(\{r\in J: g(r)\not\in V\})\nonumber\\
 &<& \varepsilon.\nonumber
\end{eqnarray}
Hence, $h\in O(U,\varepsilon)$. Therefore, $O(V,\varepsilon)\boxplus^{\bullet} f\subseteq f\boxplus^{\bullet}O(U,\varepsilon)$.  By a similar way, we can proof that $f\boxplus^{\bullet}O(V',\varepsilon)\subseteq f\boxplus^{\bullet}O(U,\varepsilon)$ for some $V'\in \mathcal{N}(\mathbf{0}).$

\smallskip
(9) Let $O(U,\varepsilon)\in \mathcal{N}(\mathbf{0}^{\bullet})$ be arbitrary. Choose an element $V\in \mathcal{N}(\mathbf{0})$ such that $\ominus V\subseteq U$. We claim that $\ominus^{\bullet}O(V,\varepsilon)\subseteq O(U,\varepsilon)$. For any $f\in \ominus^{\bullet}O(V,\varepsilon)$, we have
\begin{eqnarray}
 \mu(\{r\in J:f(r)\not\in U\})&\leq& \mu(\{r\in J:f(r)\not\in \ominus V\})\nonumber\\
 &=&\mu(\{r\in J:\ominus f(r)\not\in V\})\nonumber\\
 &<& \varepsilon.\nonumber
\end{eqnarray}
Thus, $\ominus^{\bullet}O(V,\varepsilon)\subseteq O(U,\varepsilon)$.

Therefore, $G^{\bullet}$ is a Hausdorff topological gyrogroup topology with $\mathcal{N}(\mathbf{0}^{\bullet})$ being a local base at the identity of $G^{\bullet}$.
\end{proof}

\bigskip
Recall that a space $X$ is called {\it pathwise connected} if for every $x,y\in X$, there is a continuous function $f:[0,1]\rightarrow X$ such that $f(0)=x,f(1)=y.$ Moreover, $X$ is called {\it locally pathwise connected} if for every neighbourhood $U$ of $x\in X$, there exist a pathwise connected neighbourhood $V$ such that $V\subseteq U$.

\begin{proposition}\label{3mt1}
The topological gyrogroup $G^{\bullet}$ in Theorem~\ref{tttt} is pathwise connected and locally pathwise connected.
\end{proposition}

\begin{proof}
Since every topological gyrogroup is a homogeneous space, it suffices to prove that each set $O(V,\varepsilon)$ defined in Theorem~\ref{tttt} is pathwise connected. Let $f\in O(V, \varepsilon)$ be arbitrary. We have to show that there exists a continuous mapping $\varphi:[0, 1]\rightarrow O(V,\varepsilon)$ such that $\varphi(0)=\mathbf{0}^{\bullet}$ and $\varphi(1)=f$. Indeed, by the definition of $G^{\bullet}$, there exist real numbers $a_{0}, a_{1},\ldots, a_{n}$ with $0=a_{0}<a_{1} <\ldots< a_{n}=1$ such that for
 every $k=0,1,\ldots,n-1$, $f$ is constant on $[a_{k}, a_{k+1})$. For each $t\in [0, 1]$ and for each non-negative $k<n$, let $b_{k,t}=a_{k}+t(a_{k+1}-a_{k})$. Therefore, $b_{k,0}=a_{k}, b_{k,1}=a_{k+1}$ and $a_{k}<b_{k,t}<a_{k+1}$ if $0<t<1$, for every $k= 0, 1,\ldots,n-1$. Define a mapping $\varphi:[0,1]\rightarrow G^{\bullet}$ by $\varphi(0)=\mathbf{0}^{\bullet}, \varphi(1)=f$ and, for $0<t<1$ and $0\leq r<1$,

$$\varphi(t)(r)=\left\{
\begin{matrix}
f(r),& &\mbox{ if } a_{k}\leq r<b_{k,t};\\
\mathbf{0},& &\mbox{ if } b_{k,t}\leq r<a_{k+1}.
\end{matrix}
\right.$$

Obviously, $f_{t}=\varphi(t)\in O(V, \varepsilon)$ for every $t\in [0, 1]$. It follows from the definition of $\varphi$ that $$\mu(\{r\in J: f_{t}(r)\not=f_{s}(r)\})\leq |t-s|$$ for arbitrary $s,t\in [0, 1]$. This inequality and the definition of the topology of the gyrogroup $G^{\bullet}$ given in Theorem~\ref{tttt}, together imply that the mapping $\varphi$ is continuous. Hence, every element $f\in O(V, \varepsilon)$
can be connected with the identity $\mathbf{0}^{\bullet}$ of $G^{\bullet}$ by a continuous path lying in $O(V, \varepsilon)$. This implies immediately that every two element of $O(V, \varepsilon)$ can also be connected by a continuous path inside of $O(V, \varepsilon)$, so that the set $O(V, \varepsilon)$ is pathwise connected. The similar argument applied to whole gyrogroup $G^{\bullet}$ in place of $O(V, \varepsilon)$ implies the pathwise connectedness of $G^{\bullet}$.
\end{proof}

Now we can prove our main result in this paper.

\begin{theorem}\label{3mt2}
For each topological gyrogroup $G$, there exists a natural topological
groupoid isomorphism $i_{G}: G\rightarrow G^{\bullet}$ of $G$ onto a closed subgyrogroup of the pathwise connected, locally pathwise connected topological gyrogroup $G^{\bullet}$.
\end{theorem}

\begin{proof}
We assign to each $x\in G$ the element $x^{\bullet}$ of $G^{\bullet}$ defined by $x^{\bullet}(r)=x$ for every $r\in J$. It is clear that the function $i_{G}: G\rightarrow G^{\bullet}$, where $i(x) = x^{\bullet}$ for each $x\in G$, is a topological groupoid monomorphism of $G$ to $G^{\bullet}$, and the latter gyrogroup is pathwise connected and locally pathwise connected by Proposition \ref{3mt1}.

It suffices to verify that $i_{G}(G)$ is a closed subgyrogroup of $G^{\bullet}$. Let $f\in G^{\bullet}\setminus i_{G}(G)$ be arbitrary. Thus $f$ cannot be constant as a function from $J$ to $G$. Then, we can choose real numbers $a_{1}, a_{2}, a_{3}, a_{4}$ satisfying $0\leq a_{1} < a_{2} \leq a_{3} < a_{4} \leq 1$ and two distinct elements $x_{1}, x_{2}\in G $ such that $f$ is equal to $x_{1}$ on $[a_{1}, a_{2})$ and $f$ is equal to $x_{2}$ on $[a_{3},a_{4})$. Choose an open symmetric neighbourhood $V$ of the identity in $G$ such that $(x_{1}\oplus V )\cap (x_{2}\oplus V)=\emptyset$, and let $\varepsilon=\mbox{min}\{a_{2}-a_{1},a_{4}-a_{3}\}$. We claim that $i_{G}(G)\cap (f\oplus^{\bullet}O(V,\varepsilon))=\emptyset$, that is, $f\oplus^{\bullet}g \not \in i_{G}(G)$ for every $g\in O(V,\varepsilon)$. If not, then exist $g'\in O(V,\varepsilon)$ such that $f\oplus^{\bullet} g'=x^{\bullet}$ for some $x\in G$. This implies that $g'(r)=\ominus x_{1}\oplus x$ for each $r\in[a_{1},a_{2})$, $g'(r)=\ominus x_{2}\oplus x$ for each $r\in[a_{3},a_{4})$, and $\{\ominus x_{1}\oplus x, \ominus x_{2}\oplus x\}\not\subseteq V$ (otherwise $x\in (x_{1}\oplus V )\cap (x_{2}\oplus V)\not=\emptyset$). Without loss of generality, suppose $\ominus x_{1}\oplus x\not\in V$. Put $\varepsilon=(a_{2}-a_{1})/2$, then $\mu(\{r\in J:g(r)\not\in V\})\geq a_{2}-a_{1}>\varepsilon,$
which is a contradiction with $g'\in O(V,\varepsilon)$. Thus, the complement $G^{\bullet} \setminus i_{G}(G)$ is open in $G^{\bullet}$ and, hence, $i_{G}(G)$ is closed in $G^{\bullet}$.
\end{proof}

 Next,we identify a topological gyrogroup $G$ with its image $i_{G}(G)\subseteq G^{\bullet}$ defined in Theorem \ref{3mt2}. Let us claim that the gyrogroup $G$ is placed in $G^{\bullet}$ in a very special way, permitting an extension of continuous bounded pseudometrics from $G$ over $G^{\bullet}$.

\begin{definition}
 A  pseudometric on $X$ is a function $d : X \times X \rightarrow [0, \infty)$ satisfying for all $x_{1}, x_{2}, x_{3}\in X$,

\smallskip
(1)$d(x_{1},x_{2})=0$ if $x_{1}=x_{2}$,

\smallskip
(2)$d(x_{1},x_{2})=d(x_{2},x_{1}),$

\smallskip
(3)$d(x_{1},x_{3})\leq d(x_{1},x_{2})+d(x_{2},x_{3}).$
\end{definition}

\begin{theorem}
Let $d$ be a continuous bounded pseudometric on a topological gyrogroup $G$. Then $d$ admits an extension to a continuous bounded pseudometric $d^{\bullet}$ over the gyrogroup $G^{\bullet}$. In addition, if $d$ is a metric on $G$ generating the topology of $G$, then $d^{\bullet}$ is also a metric on $G^{\bullet}$ generating the topology of $G^{\bullet}$.
\end{theorem}

\begin{proof}
Assume that $d$ is bounded without loss of generality. Take any $f,g\in G^{\bullet}$ and a partition $0=a_{0}<a_{1}<\ldots<a_{n}=1$ of $J$ such that both $f$ and $g$ are constant on each interval $J_{k}=[a_{k},a_{k+1})$ and are equal to $x_{k}$ and $y_{k}$ on this interval, respectively. We define a distance $d^{\bullet}(f,g)$ by the formula $$d^{\bullet}(f,g)=\sum^{n-1}_{k=0}(a_{k+1}-a_{k})\cdot d(x_{k},y_{k}).$$ it is clear that the number $d^{\bullet}(f,g)$ is non-negative and does not depend on the choice of the partition $a_{0},a_{1},\ldots,a_{n}$ of $J$ which keeps $f$ and $g$ constant on each $[a_{k}, a_{k+1})$. It is easy to see that the function $d^{\bullet}$ is symmetric and satisfies the triangle inequality, that is, $d^{\bullet}$ is a pseudometric on $G^{\bullet}$.

It remains to show that $d^{\bullet}$ is continuous. Let $f\in G^{\bullet}$ and $\varepsilon>0$. Take a partition $0=a_{0}<a_{1}<\ldots<a_{n}=1$ of $J$ such that $f$ has a constant value $x_{k}$ on each $J_{k}=[a_{k},a_{k+1})$. Since $d$ is continuous on $G$, there exists an open neighbourhood $V$ of the identity in $G$ such that $d(x_{k}, x_{k}\oplus y)< \varepsilon/2$ for each $y\in V$, where $k=0,1,\ldots,n-1$. Let us show that $d^{\bullet}(f, g)<\varepsilon$ for each $g\in f\oplus^{\bullet}O(V, \varepsilon/2)$. Suppose that $g\in f\oplus^{\bullet}O(V, \varepsilon/2)$, we can suppose without loss of generality that $g$ is constant on each $J_{k}$ and takes a value $y_{k}$ on this interval. Put $L$ the set of every integers $k\leq n-1$ such that $y_{k}\in x_{k}\oplus V$, and let $M=\{0,1,\ldots,n-1\}\setminus L$. It follows from the choice of $d$ and $g$ that $d(x_{k},y_{k})<\varepsilon/2$ for each $k\in L$, and that $\sum _{k\in M}(a_{k+1}-a_{k})<\varepsilon/2$. Since $d$ is bounded, and $\sum _{k\in L}(a_{k+1}-a_{k})<\varepsilon/2$, the definition of $d^{\bullet}$ implies that

\begin{eqnarray}
 d^{\bullet}(f,g)&\leq&\sum_{k\in L}(a_{k-1}-a_{k})d(x_{k},y_{k})+\sum_{k\in M}(a_{k+1}-a_{k})\nonumber\\
 &<&max_{k\in L}d(x_{k},y_{k})+\varepsilon/2\nonumber\\
 &<&\varepsilon/2+\varepsilon/2\nonumber\\
 &=&\varepsilon.\nonumber
\end{eqnarray}
This verifies the continuity of $d^{\bullet}$ on $G^{\bullet}$.

Finally, assume that $d$ is a metric on $G$ generating the topology of $G$. Thus $d^{\bullet}(f,g)>0$ for any distinct $f,g\in G^{\bullet}$, so that $d^{\bullet}$ is a metric on the set $G^{\bullet}$. Choose an arbitrary element $f\in G^{\bullet}$  and a basic open neighbourhood $f\oplus^{\bullet}O(V,\varepsilon)$ of $f$ in $G^{\bullet}$. Put $0=a_{0}<a_{1}<\cdots<a_{n}=1$ be a partition of $J$ such that $f$ takes a constant value $u_{k}$ on each $[a_{k}, a_{k+1})$. We can find an element $\delta>0$ such that $
\{u\in G:d(u_{k},u)<\delta\}\subseteq u_{k}\oplus V$, for every $k=0,1,\cdots,n-1$. Therefore, if $k<n$ and $y\in G\setminus (u_{k}\oplus V)$, thus $d(u_{k},y)\geq \delta$. Put $\delta_{0}=\varepsilon \cdot \delta$. We could show that $$\{g\in G^{\bullet}:d^{\bullet}(f,g)<\delta_{0}\}\subseteq f\oplus^{\bullet}O(V,\varepsilon).$$
Indeed, assume that $g\in G^{\bullet}$ meets $d^{\bullet}(f,g)<\delta_{0}$. There exists a partition $0=b_{0}<b_{1}<\ldots<b_{N}=1$ of $J$ refining the partition $0=a_{0}<a_{1}<\ldots<a_{n}=1$ such that $g$ has a constant value $y_{i}$ on each $[b_{i},b_{i+1})$, for $i=0,\ldots,N-1$. Then $f$ is also constant on each $[b_{i}, b_{i+1})$, and let $f(b_{i})=x_{i}$. Denote by $P$ the set of all non-negative integers $i<N$ such that $y_{i}\not\in x_{i}\oplus V$. Clearly, $$\sum_{i\in P}(b_{i+1}-b_{i})d(x_{i},y_{i})\leq \sum_{0\leq i< N}(b_{i+1}-b_{i})d(x_{i},y_{i})=d^{\bullet}(f,g)<\delta_{0}.$$ Since $d(x_{i},y_{i})\geq \delta$ if $y_{i}\in G \setminus (x_{i}\oplus V)$, we have that $$\delta\cdot\sum_{i\in P}(b_{i+1}-b_{i})\leq \sum_{i\in P}(b_{i+1}-b_{i})d(x_{k},y_{k})<\delta_{0}.$$ It follows that $$\sum_{i\in P}(b_{i+1}-b_{i})<\delta_{0}/\delta =\varepsilon.$$ The definition of $P$ implies the equality $$D=\{r\in J:g(r)\not\in f(r)\oplus V\}=\bigcup\{[b_{i},b_{i+1}):i\in P\}.$$
It follows that $\mu(D)<\varepsilon$ , whence $\ominus f\oplus g\in O(V,\varepsilon)$ or, equivalently, $g\in f\oplus^{\bullet}O(V,\varepsilon)$. This proves the inclusion $$\{g\in G^{\bullet}:d^{\bullet}(f,g)<\delta_{0}\}\subseteq f\oplus^{\bullet}O(V,\varepsilon).$$ Since the sets $f\oplus^{\bullet}O(V,\varepsilon)$ form a base of $G^{\bullet}$ at $f$ and the metric $d^{\bullet}$ is continuous, it follows that $d^{\bullet}$ generates the topology of the gyrogroup $G^{\bullet}$.
\end{proof}

\begin{corollary}
Let $f$ be a continuous real-valued bounded function on a topological gyrogroup $G$. Then $f$ admits an extension to a bounded continuous function on the gyrogroup $G^{\bullet}$.
\end{corollary}

The construction of the gyrogroup $G^{\bullet}$  allows to extend continuous groupoid homomorphisms.

\begin{proposition}
Let $\varphi:G\rightarrow H$ be a continuous groupoid homomorphism of topological
gyrogroups. Then $\varphi$ admits a natural extension to a continuous groupoid homomorphism $\varphi^{\bullet}:G^{\bullet}\rightarrow H^{\bullet}$. In addition, if $\varphi$ is open and onto, then so is $\varphi^{\bullet}$.
\end{proposition}

\begin{proof}
For each $f\in G^{\bullet}$, define $\varphi^{\bullet}(f)\in H^{\bullet}$ by $\varphi^{\bullet}(f)(r)=\varphi(f(r))$ for each $r\in J=[0,1)$. If $f,g\in G^{\bullet}$ and $r\in J$, thus $$\varphi^{\bullet}(f\oplus^{\bullet}g)(r)=\varphi(f(r)\oplus g(r))=\varphi(f(r))\oplus \varphi (g(r))=[\varphi^{\bullet}(f)\oplus\varphi^{\bullet}(g)](r).$$
Therefore, $\varphi^{\bullet}(f\oplus^{\bullet}g)=\varphi^{\bullet}(f)\oplus\varphi^{\bullet}(g)$, and we conclude that $\varphi^{\bullet}:G^{\bullet}\rightarrow H^{\bullet}$ is a groupoid homomorphism.

Let us show that $\varphi^{\bullet}$ is continuous, choose an open neighbourhood $V$ of the identity in $H$ and a real number $\varepsilon>0$. By the continuity of $\varphi$, we can find an open neighbourhood $U$ of the identity in $G$ such that $\varphi(U)\subseteq V$. Hence the definition of $\varphi^{\bullet}$ implies immediately
that $\varphi^{\bullet}(O(U,\varepsilon))\subseteq O(V,\varepsilon)$, which proves that $\varphi^{\bullet}$ is continuous. It is easy to see that $\varphi^{\bullet}\circ i_{G}=i_{H}\circ\varphi$, where $i_{G}:G\rightarrow G^{\bullet}$ and $i_{H}:H\rightarrow H^{\bullet}$ are natural topological groupoid monomorphisms defined in Theorem \ref{3mt2}. In particular, identifying $G$ with $i_{G}(G)$ and $H$ with $i_{H}(H)$, the above equality takes the form $\varphi^{\bullet}\upharpoonright G=\varphi$, i.e., $\varphi^{\bullet}$ is a continuous extension of $\varphi$ over $G^{\bullet}$.

Finally, assume that the groupoid homomorphism $\varphi$ is open and $\varphi(G)=H$. We may assume that $\varphi^{\bullet}(G^{\bullet})= H^{\bullet}$ and $\varphi^{\bullet}(O(V,\varepsilon))=O(W,\varepsilon)$ for every $V\in \mathcal{N}_{G}(e)$ and $\varepsilon>0$, where $W=\varphi(V)$ is an open neighbourhood of the identity in $H$. For any $h^{\bullet}\in H^{\bullet}$, thus $\varphi (g)=h$ for some $g\in G$. It is easy to verify that $\varphi^{\bullet}(g^{\bullet})= h^{\bullet}$. Indeed, $$\varphi^{\bullet}(g^{\bullet})(r)=\varphi(g(r))=\varphi(g)=h=h^{\bullet}(r),$$ thus $\varphi^{\bullet}(g^{\bullet})= h^{\bullet}$. Since
\begin{eqnarray}
 f\in O(V,\varepsilon)&\Leftrightarrow& \mu(\{r\in J: f(r)\not\in V\})<\varepsilon\nonumber\\
 &\Leftrightarrow& \mu(\{r\in J: \varphi(f(r))\not\in \varphi(V)\})<\varepsilon\nonumber\\
 &\Leftrightarrow& \mu(\{r\in J: \varphi^{\bullet}(f)(r)\not\in W\})<\varepsilon\nonumber\\
 &\Leftrightarrow& \varphi^{\bullet}(f)\in O(W,\varepsilon),\nonumber
\end{eqnarray}
 then $\varphi^{\bullet}(O(V,\varepsilon))=O(W,\varepsilon)$. Therefore, $\varphi^{\bullet}$ is open and onto.
\end{proof}

The correspondences $G \mapsto G^{\bullet}$ and $\varphi \mapsto \varphi ^{\bullet}$ are functorial since the equality $(\psi\circ\varphi)^{\bullet}=\psi^{\bullet}\circ \varphi^{\bullet}$ holds for any continuous groupoid homomorphisms $\varphi:G\rightarrow H$ and $\psi:H\rightarrow K$. This defines the covariant functor $^{\bullet}$ in the category of topological gyrogroups and continuous groupoid homomorphisms.

\begin{proposition}
Let $G$ be a topological gyrogroup and $H$ a subgyrogroup of $G$. If $\varphi$ is the identity embedding of $H$ to $G$, then the natural homomorphic extension $\varphi^{\bullet}:H^{\bullet}\rightarrow G^{\bullet}$ is a topological groupoid monomorphism.
\end{proposition}
\begin{proof}
\smallskip
It is clear that $\varphi^{\bullet}$ is a groupoid monomorphism. Let $V$ be an arbitrary open
neighbourhood of the neutral element in $G$. Put $W=H\cap V$. For any $\varepsilon>0$, we have that $(\phi^{\bullet})^{-1}(O(V,\varepsilon))=O(W,\varepsilon)$.
\end{proof}

Finally, we prove that if a topological gyrogroup is metrizable, separable and $\sigma$-compact then $G^{\bullet}$ is metrizable, separable and $\sigma$-compact respectively. First, note that it was proved in \cite{BL3} that for any cardinality $\kappa > \omega$, there exists a gyrogroup $G$
with subgyrogroup $H$ of the cardinality $\kappa$ such that $H$ is not a group.

\begin{definition}\cite{GG} Let $\mathcal{P}$ be a family of subsets of a space $X$. The family $\mathcal{P}$ is called a {\it network} if for each $x\in X$ and any open neighborhood $U$
of $x$ there exists $P\in \mathcal{P}$ such that $x\in P\subseteq U$.
\end{definition}

\begin{definition}\cite{BLL} A topological gyrogroup $G$ is called {\it left (right) $\kappa$-narrow} if, for each neighborhood $U$ of $\mathbf{0}$ in $G$, we can find a subset $K \subseteq G$ with $|K| \leq \tau$ satisfying $K \oplus U = G$ $(U \oplus K = G)$ for an infinite cardinality $\kappa$ . We call $G$ {\it $\kappa$-narrow} if it is both left $\kappa$-narrow and right $\kappa$-narrow.
\end{definition}

\begin{theorem}
Let $\kappa$ be an infinite cardinal number, and let $G$ be a topological gyrogroup
having one of the following properties:

\smallskip
(1) $G$ is metrizable;

\smallskip
(2) $G$ has a local base at the identity of cardinality $\leq \kappa$;

\smallskip
(3) $G$ has a network of cardinality $\leq \kappa$;

\smallskip
(4) $G$ has a dense subset of cardinality $\leq \kappa$;
\smallskip

(5) $G$ is $\kappa$-narrow.

Then the gyrogroup $G^{\bullet}$ has the same property.
\end{theorem}

\begin{proof}
Since, the metrizability of a topological gyrogroup is equivalent
to having a countable base at the identity\cite{CZ}, item (1) follows from (2) with $\kappa=\omega$. Hence, assume that $G$ has a local base $\mathcal{N}$ at the identity $\mathbf{0}$ satisfying $|\mathcal{N}|\leq \kappa$. Therefore, according to  the definition of the topology of $G^{\bullet}$, the family $\mathcal{N}^{\bullet}=\{O(V,1/n):V\in\mathcal{N},n\in \mathbb{N}\}$ is a local base at the identity $\mathbf{0}^{\bullet}$ of $G^{\bullet}$, and $|\mathcal{N}^{\bullet}|\leq |\mathcal{N}|\cdot\omega\leq\kappa$. This follows the assertion of the theorem for (1) and (2).

For (4), choose a dense set $D\subseteq G$ with $|D|\leq \kappa$. Denoted by $S$ the set of all $f\in G^{\bullet}$ for which there exist rational numbers $b_{0},b_{1},\cdots,b_{m}$ with $0=b_{0}<b_{1}<\cdots<b_{m}=1$ such that $f$ is constant on each semi-open interval $J_{k}=[b_{k},b_{k+1})$ and takes  $x_{k}\in D$ on $J_{k}$. It is easy to see that $|S|\leq|D|\cdot\omega\leq\kappa$, and let us show that $S$ is dense in $G^{\bullet}$. Indeed, let $f\oplus^{\bullet}O(V,\varepsilon)$ be a basic open neighbourhood of $f\in G^{\bullet}$, where $V$ is an open neighbourhood of $\mathbf{0}$ in $G$ and $\varepsilon >0$. Thus we can find $0=a_{0}<a_{1}<\cdots<a_{n}=1$ such that the function $f$ is constant on $[a_{k},a_{k+1})$ for every $k<n$. Choose rational numbers $b_{1},\cdots,b_{n-1}$ in $J$ such that $a_{k}\leq b_{k}<a_{k+1}$ for every $k<n$ and $\sum ^{n-1}_{k=1}(b_{k}-a_{k})<\varepsilon$. Also, put $b_{0}=0$ and $b_{n}=1$. For each $k<n$, take a point $y_{k}\in D\cap (x_{k}\oplus V)$, where $x_{k}=f(a_{k})$, and define an element $g\in S$ by letting $g(r)=y_{k}$ for any $r\in [b_{k},b_{k+1}), k=0,\cdots,n-1$. It follows that $g\in f\oplus^{\bullet}O(V,\varepsilon)$, hence $S$ is dense in $G^{\bullet}$ and $|S|\leq\kappa$.

In the case of (3), take a network $\mathcal{P}$ for $G$ with $|\mathcal{P}|\leq\kappa$. For each $m\in \mathbb{N}$, let $J(m)$ is the set of all $m$-tuples $(b_{1},\cdots,b_{m})$ of rational numbers such that $0<b_{1}<\cdots<b_{m}=1$. Given $m,n\in\mathbb{N}$, an element $\vec{b}=(b_{1},\cdots,b_{m})\in J(m)$ and $\vec{P}=(P_{1},\cdots,P_{m})\in \mathcal{P}^{m}$, denote a subset $Q(m, n,\vec{b}, \vec{p})$ of $G^{\bullet}$ as the set of every $g\in G^{\bullet}$ such that the measure (with respect to the Lebesgue measure $\mu$ on $J$) of the set of all $r\in J$ satisfying $b_{k}\leq r<b_{k+1}$ and $g(r)\not \in P_{k+1}$, for some $k=0,1,\cdots,m-1$, is less than $1/n$ (we always put $b_{0}=0$). Then the family $\mathcal{A}$ of all sets $Q(m,n,\vec{b},\vec{p})$ with $m,n\in \mathbb{N}, \vec{b}\in J(m)$ and $\vec{P}\in \mathcal{P}^{m}$ has the cardinality less than or equal to $\kappa$, and we have to show that $\mathcal{A}$ is a network for $G^{\bullet}$.

Indeed, take arbitrary $f\in G^{\bullet}$ and $O(V,\varepsilon)\ni \mathbf{0}^{\bullet}$. There exist numbers $0=a_{0}<a_{1}<\cdots<a_{m}=1$ such that $f$ is constant on each semi-open interval $[a_{k},a_{k+1})$. For every $k=0,1,\cdots,m-1$, put $x_{k}=f(a_{k})$ and choose an element $P_{k+1}\in \mathcal{P}$ such that $x_{k}\in P_{k+1}\subseteq x_{k}\oplus V$. Then take $n\in \mathbb{N}$ with $1/n<\varepsilon$ and choose $\vec{b}=(b_{1},\cdots,b_{m})\in J(m)$ such that $a_{k}\leq b_{k}<a_{k+1}$ for each $k=1,\cdots,m-1$, and $\sum^{m-1}_{k=1}(b_{k}-a_{k})<1/(2n)$. All we need to proof that $f\in Q(m, 2n,\vec{b}, \vec{p})\subseteq f\oplus^{\bullet}O(V,\varepsilon)$, where $\vec{p}=(P_{1},\cdots,P_{m})\in \mathcal{P}^{m}$. Let $b_{0}=0$ and $b_{m}=1$.

It follows from the choice of $\vec{b}$ and $\vec{p}$ that $f(r)=x_{k}\in P_{k+1}$ for each $r\in [b_{k},a_{k+1})$, where $k=0,1,\cdots,m-1$. Hence, if $b_{k}\leq r<b_{k+1}$ and $f(r)\not\in P_{k+1}$ for some $k<m$, then $r\in [a_{k+1},b_{k+1})$ and $k+1\not=m$. Since $\sum^{m-1}_{k=1}(b_{k}-a_{k})<1/(2n)$, it follows that $f\in Q(m,2n,\vec{b},\vec{p})$. we could show that $Q(m,2n,\vec{b},\vec{p})\subseteq f\oplus^{\bullet}O(V,\varepsilon)$, it suffices to verify that $(\ominus^{\bullet} f)\oplus^{\bullet} g\in O(V,\varepsilon)$ for every $g\in Q(m,2n,\vec{b},\vec{p})$. Let $g\in Q(m,2n,\vec{b},\vec{p})$ be arbitrary. It follows from the definition of $Q(m,2n,\vec{b},\vec{p})$ that the set $$L=\{r\in J:b_{k}\leq r<b_{k+1} \mbox{ and } g(r)\not\in P_{k+1} \mbox{ for some } k<m\}$$
satisfies $\mu(L)<1/(2n)$. If $r\in (J\setminus L)$ and $b_{k}\leq r<a_{k+1}$ for some $k<m$, then $g(r)\in P_{k+1}$ and $f(r)=x_{k}$, whence it follows that $((\ominus^{\bullet} f)\oplus ^{\bullet}g)(r)=(\ominus x_{k})\oplus g(r)\in x_{k}\oplus P(k+1)\subseteq V$. This implies that $$M=\{r\in J:((\ominus^{\bullet} f)\oplus ^{\bullet}g)(r)\not\in V\}\subseteq L\cup\bigcup^{m-1}_{k=1}[a_{k},b_{k}),$$ so that $\mu(M)<1/(2n)+1/(2n)=1/n$. Hence, $(\ominus f)\oplus g\in O(V,\varepsilon)$ and $g\in f\oplus^{\bullet}O(V,\varepsilon)$, as claimed. This proves that $\mathcal{A}$ is a network for the gyrogroup $G^{\bullet}$.

Finally, assune that the gyrogroup $G$ is $\kappa$-narrow. We have to show that the gyrogroup $G^{\bullet}$ is left $\kappa$-narrow. Let $O=O(V,\varepsilon)$ be a basic neighbourhood of $\mathbf{0}^{\bullet}$ in $G^{\bullet}$, where $V$ is an open neighbourhood of $\mathbf{0}$ in $G$ and $\varepsilon$ is a positive real number. Thus we can find  a set $D\subseteq G$ such that $G=D\oplus V$ and $|D|\leq \kappa$. Denote by $S$ the set of elements $g\in G^{\bullet}$ such that, for some rational numbers $b_{0},b_{1},\cdots,b_{n}$ with $0=b_{0}<b_{1}<\cdots<b_{n}=1$, the function $g$ is constant on each $J_{k}=[b_{k},b_{k+1})$, and the value of $g$ on $J_{k}$ is an element of $D$. It is clear that $|S|\leq \kappa$. It remains to show that $G^{\bullet}=S\oplus^{\bullet}O$. For any  $f\in G^{\bullet}$, one can find $a_{0},a_{1},\cdots,a_{n}$ with $0=a_{0}<a_{1}<\cdots<a_{n}=1$ such
that $f$ is constant on each $[a_{k},a_{k+1})$. Similarly to the construction in (4), take rational numbers $b_{0},b_{1},\cdots,b_{n}$ such that $b_{0}=0,b_{n}=1, a_{k}\leq b_{k}< a_{k+1}$, for $k=1,\cdots,n-1$, and $\sum^{n-1}_{k=1}(b_{k}-a_{k})<\varepsilon$. For every $k\in \{0,1,\cdots,n-1\}$, choose an element $x_{k}\in D$ such that $f(a_{k})\in x_{k}\oplus V$. Denote by $g$ the element of $G^{\bullet}$ which is constant on each $J_{k}=[b_{k},b_{k+1})$ and takes the value $x_{k}$ on $J_{k}$. Hence, $g\in S$, and it is easy to verify that $f\in g\oplus^{\bullet}O$. This proves the equality $G^{\bullet}=S\oplus^{\bullet}O$ and and implies that the gyrogroup $G^{\bullet}$ is left $\kappa$-narrow. Similarly, we can proof that the gyrogroup $G^{\bullet}$ is right $\kappa$-narrow. Therefore, $G^{\bullet}$ is $\kappa$-narrow.
\end{proof}

Recall that a space $X$ is called {\it $\sigma$-compact} if it can be expressed as the union of countable compact sets.

\begin{theorem}
The gyrogroup $G^{\bullet}$ is $\sigma$-compact if and only if $G$ is  $\sigma$-compact.\label{3mt11}
\end{theorem}

\begin{proof}
 The necessary is obvious since $G$ is topologically groupoid isomorphic to a closed
subgyrogroup of $G^{\bullet}$, by Theorem \ref{3mt2}. Conversely, let $G$ be the union of compact sets $K_{i}$, where $i\in \omega$. We can assume that $K_{i}\subseteq K_{i+1}$ for each $i\in \omega$. Let $I$ be the closed unit segment with usual interval topology. For every $n,m\in \mathbb{N}$, let $$A_{n}=\{(a_{1},\cdots,a_{n})\in I^{n}:0<a_{1}<\cdots<a_{n}<1\}$$ and $$A_{n,m}=\{(a_{1},\cdots,a_{n})\in A_{n}:a_{k+1}-a_{k}\geq 1/m \mbox{ for each }\ k\leq n\},$$
where $a_{0}=0$ and $a_{n+1}=1$. It is clear that $A_{n}=\bigcup^{\infty}_{m=1}A_{n,m}$ and that each $A_{n,m}$ is closed in $I^{n}$. In particular, the sets $A_{n,m}$ are compact.

Given $n\in\mathbb{N}$, we define a mapping $\varphi_{n}:G_{n+1}\times A_{n}\rightarrow G^{\bullet}$ by the rule $$\varphi_{n}(x_{0},\cdots,x_{n},a_{1},\cdots,a_{n})=f,$$ where the function $f:J\rightarrow G$ takes the constant value $x_{k}$ on $[a_{k},a_{k+1})$ for each $k\leq n$. We claim that the restriction of $\varphi_{n}$ to $G_{n+1}\times A_{n,m}$ is continuous for each $m\in \mathbb{N}$. Indeed, take $p=(x_{0},\cdots,x_{n}, a_{1},\cdots,a_{n})\in G_{n+1}\times A_{n,m}$ and put $f=\varphi_{n}(p)$. Consider a
basic open neighbourhood $f\oplus^{\bullet}O(V,\varepsilon)$ of $f$ in $G^{\bullet}$, where $V$ is an open neighbourhood of the identity in $G$ and $\varepsilon>0$. Choose a positive number $\delta<\mbox{min}\{\varepsilon/(2n),1/(2m)\}$, and define a neighbourhood $W$ of $p$ in $G_{n+1}\times R_{n}$ by $$W=x_{0}\oplus V\times\l\cdots\times x_{n}\oplus V\times(a_{1}-\delta,a_{1}+\delta)\times\cdots\times (a_{n}-\delta,a_{n}+\delta).$$

Let us show that $\varphi_{n}(q)\in f\oplus^{\bullet}O(V,\varepsilon)$, for each $q\in W\cap (G_{n+1}\times A_{n,m})$. Clearly, $q=(y_{0},\cdots,y_{n},b_{1},\cdots,b_{n})$, where $(y_{0},\cdots,y_{n})\in G_{n+1}$ and $(b_{1},\cdots,b_{n})\in A_{n,m}$. Set
$g=\varphi_{n}(q)$. Clearly, $(\ominus x_{k})\oplus y_{k}\in V$ for each $k\leq n$, and if $r\in J\setminus \bigcup^{n}_{k=1}(a_{k}-\delta,a_{k}+\delta)$ and $a_{k}\leq r<a_{k+1}$ for some $k\leq n$, then $b_{k}\leq r<b_{k+1}$ and $g(r)=y_{k}$ (again, we put $b_{0}=0$
and $b_{n+1}=1$). Hence, $(\ominus f(r))\oplus g(r)=(\ominus x_{k})\oplus y_{k}\in V$. This implies that $$L=\{r\in J:(\ominus^{\bullet}f\oplus^{\bullet}g)(r)\not\in V\}\subseteq \bigcup^{n}_{k=1}(a_{k}-\delta,a_{k}+\delta),$$
so that $\mu(L)\leq 2n\delta<\varepsilon$. Therefore, $(\ominus^{\bullet}f)\oplus^{\bullet}g\in O(V,\varepsilon)$. We conclude that $g=\varphi_{n}(q)$ is an element of $f\oplus^{\bullet}O(V,\varepsilon)$, that is, $\varphi_{n}$ is continuous on $G_{n+1}\times A_{n,m}$. It is easy to see that $G^{\bullet}=\sum^{\infty}_{i,n,m=1}\varphi_{n}(K^{n+1}_{i}\times A_{n,m})$, where each image $\varphi_{n}(K^{n+1}_{i}\times A_{n,m})$ is a compact subset of $G^{\bullet}$ by the continuity of $\varphi_{n}$ on the product space $(G^{n+1}_{i}\times A_{n,m})$. This
proves that the group $G^{\bullet}$ is $\sigma$-compact.
\end{proof}

Combining Theorems \ref{3mt2} and \ref{3mt11}, we deduce the following result.

\begin{corollary}
Every $\sigma$-compact gyrogroup is topologically groupoid isomorphic to a closed
subgyrogroup of a $\sigma$-compact, pathwise connected, locally pathwise connected gyrogroup.
\end{corollary}


\begin{thebibliography}{99}
\bibitem{AA} A.V. Arhangel' ski\v\i, M. Tkachenko, {\it Topological Groups and Related Structures}, Atlantis Press and World Sci., 2008.

\bibitem{AW} W. Atiponrat, {\it Topological gyrogroups: generalization of topological groups}, Topol. Appl., 224 (2017): 73--82.

\bibitem{BL} M. Bao, F. Lin, {\it Feathered gyrogroups and gyrogroups with countable pseudocharacter}, Filomat, 33 (16)(2019): 5113--5124.

\bibitem{BL3} M. Bao, F. Lin, {\it Submaximal properties in (strongly) topological gyrogroups}, Filomat, 35 (7)(2021): 4912.

\bibitem{BLL} M. Bao, Y. Lin, F. Lin,{\it Strongly Topological Gyrogroups with Remainders Close to Metrizable}, Bulle. Iran. Math. Soc., (2021) DOI:10.1007/s41980-021-00594-8.

\bibitem{BZX} M. Bao, X. Zhang, X. Xu, {\it Separability in (strongly) topological gyrogroups}, Filomat (2021) Accepted, https://arxiv.org/abs/2011.02633.

\bibitem{CZ} Z. Cai, S. Lin, W. He, {\it A note on Paratopological Loops}, Bull. Malay. Math. Sci. Soc., 42(5)(2019): 2535--2547.

\bibitem{GG} G. Gruenhage, {\it Generalized metric spaces, in: K. kunen, J.E. Vaughan (Eds), Handbook of Set-Theoretic Topology }, Elsevier Science Publishers B.V., Amsterdam, 1984, pp. 423--501.

\bibitem{JX} Y. Jin, L. Xie, {\it On paratopological gyrogroups}, https://arxiv.org/abs/2104.00956v1.

\bibitem{ST} T. Suksumran, K. Wiboonton, {\it Isomorphism theorems for gyrogroups and L-subgyrogroups}, J. Geom. Symmetry Phys., 37 (2015): 67--83.

\bibitem{UA} A.A. Ungar, {\it Analytic Hyperbolic Geometry and Albert Einstein's Special Theory of Relativity}, World Scientific, Hackensack, New Jersey, 2008.

\bibitem{WAS2020} J. Wattanapan, W. Atiponrat, T. Suksumran, {\it Embedding of locally compact Hausdorff topological
gyrogroups in topological groups}, Topol. Appl. 273(2020): 107102.


\end{thebibliography}
\end{document}